\documentclass{amsart}

\usepackage{amssymb}
\usepackage{amsmath}
\usepackage{amsthm}
\usepackage{nicefrac}

\makeatletter
\def\@begintheorem#1#2[#3]{%
    \def\naam{#1}
  \deferred@thm@head{\the\thm@headfont \thm@indent
    \@ifempty{#1}{\let\thmname\@gobble}{\let\thmname\@iden}%
    \@ifempty{#2}{\let\thmnumber\@gobble}{\let\thmnumber\@iden}%
    \@ifempty{#3}{\let\thmnote\@gobble}{\let\thmnote\@iden}%
    \thm@swap\swappedhead\thmhead{#1}{#2}{#3}%
    \the\thm@headpunct
    \thmheadnl 
    \hskip\thm@headsep
  }%
  \ignorespaces}
\makeatother

\def\voegToe#1#2#3{\immediate\write1{\string\newlabel{#1}{{#2}{#3}}}}
\def\thlabel#1{\voegToe{#1}{\naam\noexpand~\thetheorem}{\thepage}}

\newtheorem{theorem}{Theorem}[section]
\newtheorem{proposition}[theorem]{Proposition}
\newtheorem{lemma}[theorem]{Lemma}
\newtheorem{corollary}[theorem]{Corollary}
\newtheorem{question}[theorem]{Question}

\theoremstyle{remark}

\theoremstyle{remark}

\newtheorem*{claim2}{\sc Claim}
\newenvironment{claim}{\begin{claim2}\rm}{\end{claim2}\rm}

\newenvironment{pfclaim}{\begin{trivlist}\item[]{\sc Proof of
Claim.}}{\mbox{~}\hfill
{\mbox{$\blacktriangleleft$}}\end{trivlist}}



\newcommand{\Star}[2]{{\rm St}(#1,#2)}
\newcommand{\st}[3][.]{\if#1.{\mbox{St}(#2,\!~#3)}\else{\mbox{St}^{#1}(#2,\!~#3)}\fi}

\title{Weak extent, submetrizability and diagonal degrees}

\author{D. Basile}
\address{Universit\`a degli Studi di Catania\\
                Dipartimento di Matematica e Informatica\\
                Viale Andrea Doria 6\\
                95125 Catania\\
                Italy}
\email{basile@dmi.unict.it}

\author{A. Bella}
\address{Universit\`a degli Studi di Catania\\
                Dipartimento di Matematica e Informatica\\
                Viale Andrea Doria 6\\
                95125 Catania\\
                Italy}
\email{bella@dmi.unict.it}

\author{G. J. Ridderbos}
\address{Faculty of Electrical Engineering\\
         Mathematics and Computer Science\\
         TU Delft\\
         Postbus 5031\\
         2600~GA {} Delft\\
         the Netherlands}
\email{G.F.Ridderbos@tudelft.nl}
\urladdr{http://aw.twi.tudelft.nl/~ridderbos}

\newcommand{\sier}[1]{\mathcal{#1}}
\newcommand{\closure}[1]{\overline{#1}}
\newcommand{\pichar}[1]{\pi\chi(#1)}

\begin{document}

\begin{abstract}
We show that if $X$ has a zero-set diagonal and $X^2$ has
countable weak extent, then $X$ is submetrizable. This
generalizes earlier results from Martin and Buzyakova.
Furthermore we show that if $X$ has a regular $G_\delta$-diagonal
and $X^2$ has countable weak extent, then $X$ condenses onto a
second countable Hausdorff space. We also prove several
cardinality bounds involving various types of diagonal degree.
\end{abstract}

\date{\today}

\maketitle

\section{Introduction}

A space is called submetrizable if it admits a coarser metrizable
topology. The diagonal of $X^2$, denoted by $\Delta_X$, is the
set $\{(x,x) : x\in X\}$. A space $X$ is said to have a zero-set
diagonal if there is a continuous function $f:X^2\to [0,1]$ such
that $\Delta_X = f^{-1}(0)$ and $X$ is said to have a regular
$G_\delta$-diagonal if $\Delta_X$ is a regular $G_\delta$-subset
of $X$, i.e. it is the intersection of countably many closed
neighbourhoods.

It is well-known that every submetrizable space has a zero-set diagonal, but the converse is false in general (see the example constructed in \cite{reed} and the remarks on it made in \cite[Example 2.17]{arhanbuzya}). This suggests to find conditions for a space with a zero-set diagonal to be submetrizable.

For example, in \cite{mar75} H.W. Martin proved that separable spaces having a
zero-set diagonal are submetrizable. In another direction, in
\cite{buz05} R.Z. Buzyakova showed that if $X$ has a zero-set
diagonal and $X^2$ has countable extent then $X$ is
submetrizable. Separability and countable extent are independent
properties, but they have a quite natural common weakening, namely countable weak extent. 
In the first part of our paper, we give a simultaneous
generalization of both the previous  results by showing that
spaces
having a zero-set diagonal and whose square has countable weak
extent are submetrizable.

Buzyakova also proved (see \cite[Theorem 2.4 \& 2.5]{buz05})
that if $X$ has a regular $G_\delta$-diagonal and either it is
separable or $X^2$ has countable extent, then $X$ condenses onto
a second-countable Hausdorff space. Again, we give a simultaneous
generalization of both these results by showing that if $X^2$ has
countable weak extent and a regular $G_\delta$-diagonal, then $X$
condenses onto a second-countable Hausdorff space.

In the second part of the paper we will study cardinality
bounds on a space $X$ according to the specific  way its 
diagonal is embedded in $X^2$.

\section{Notation and terminology}

For all undefined notions we refer to \cite{engelking}.

Recall that $X$ condenses onto $Y$ if there is a continuous
bijection from $X$ onto $Y$. So a space is submetrizable if and
only if it condenses onto a metrizable space. The extent of a
space $X$, denoted by $e(X)$, is the supremum of the
cardinalities of closed and discrete subsets of $X$. The weak
extent of a space $X$, denoted by $we(X)$, is the least cardinal
number $\kappa$ such that for every open cover $\mathcal{U}$ of
$X$ there is a subset $A$ of $X$ of cardinality no greater than
$\kappa$ such that $\st{A}{\mathcal{U}}=X$. It is clear that
$we(X)\leq d(X)$ and $we(X)\leq e(X)$. Note that spaces with
countable weak extent are called star countable by several
authors (see, for instance \cite{alas11}). For a space $X$ the
weak-Lindel\"of number of $X$, denoted by $wL(X)$, is the least
cardinal $\kappa$ such that every open cover of $X$ has a
subfamily of cardinality no greater than $\kappa$ whose union is
dense in $X$.

Whenever $\sier{B}$ is a collection of subsets of $X$ and
$A\subseteq X$, the star at $A$ with respect to $\sier{B}$,
denoted by $\st{A}{\sier{B}}$, is defined by the formula
$$
\st{A}{\sier{B}} = \bigcup\{ B\in\sier{B} : A\cap
B\not=\emptyset\}.
$$
If we let $\st[0]{A}{\sier{B}} = A$ then, for $n\in\omega$, the
$n$-star around $A$ is defined by induction:
$$
\st[n+1]{A}{\sier{B}} = \st{\st[n]{A}{\sier{B}}}{\sier{B}}.
$$
Note that $\st[1]{A}{\sier{B}}=\st{A}{\sier{B}}$. If $A=\{a\}$ we
write $\st[n]{a}{\sier{B}}$ instead of $\st[n]{A}{\sier{B}}$.

If $n\in\omega$, and $\kappa$ is an infinite cardinal, we say
that a space $X$ has a rank $n$ $G_\kappa$-diagonal (a strong
rank $n$ $G_\kappa$-diagonal) if there is a sequence
$\{\mathcal{U}_\alpha : \alpha < \kappa\}$ of open covers of $X$
such that for all $x\not=y$, there is some $\alpha < \kappa$ such
that $y\not\in St^{n}(x,\mathcal{U}_\alpha)$
($y\not\in\overline{St^{n}(x,\mathcal{U}_\alpha})$). When
$\kappa=\omega$, we will simply write rank $n$-diagonal.
We will denote the minimal cardinal $\kappa$ such that $X$ has a
rank $n$ $G_\kappa$-diagonal or a strong rank $n$
$G_\kappa$-diagonal by $\Delta_n(X)$ and $s\Delta_n(X)$,
respectively.  The formula 
$\Delta_n(X)\leq\min\{\Delta_{n+1}(X),s\Delta_n(X)\}$ is
obviously true. If $n=1$
we will omit the number $1$.

Recall that a space has a $G_\delta$-diagonal if and only if it
has a rank $1$-diagonal (this was proved by Ceder in \cite[Lemma
5.4]{ced61}). In analogy to Ceder's result, Zenor proved in
\cite[Theorem 1]{zen72} that a space $X$ has a regular
$G_\delta$-diagonal if and only if there is a sequence
$\{\mathcal{U}_n : n\in\omega\}$ of open covers of $X$ such that
for all $x\not=y$, there is a neighbourhood $U$ of $x$ and some
$n\in\omega$ such that $y\not\in\overline{St(U,\mathcal{U}_n)}$. 

In particular, if a space has a strong rank $2$-diagonal, then it
has a regular $G_\delta$-diagonal. We must say that at present we
do not know any example of spaces having a regular
$G_\delta$-diagonal that does not have a strong rank
$2$-diagonal. Even more intriguing is the relationship between
regular $G_\delta$-diagonal and rank $2$-diagonal. It is
well-known that there exists a space with a rank $2$-diagonal
that does not have a regular $G_\delta$-diagonal, namely the
Mrowka space $\Psi$ (see \cite{arhanbuzya}). This easily follows
from a result of McArthur (\cite{mcarthur}), stating that a
pseudocompact space with a regular $G_\delta$-diagonal is
metrizable. But the following question from A. Bella
(\cite{bella87}) is still open:

\begin{question}
Does any space with a regular $G_\delta$-diagonal have a rank $2$-diagonal?
\end{question}

A good reason  for asking such a question comes out from  a
comparison of the following two facts. In
\cite{bella87} Bella proved that a ccc space  with a rank
$2$-diagonal has cardinality not exceeding $2^\omega$.  
Much more recently and with a certain effort, in \cite{buz06} 
Buzyakova has shown that a ccc space with a regular
$G_\delta$-diagonal has again cardinality not exceeding
$2^\omega$. Therefore, a positive answer to the previous question
would imply a trivial proof of the latter
result from the former. 

\section{Zero-set diagonal vs  submetrizability}

The aim of this section is to provide a simultaneous
generalization of Martin and Buzyakova's results. The obvious way
to accomplish this is by using the weak extent. However, we
actually present a formally stronger result obtained by means of
an even weaker form of the weak extent of a square. 

The weak double extent of a space $X$, denoted by $wee(X)$, is
the smallest cardinal $\kappa $ such that whenever $\mathcal{U}$
is an open cover of $X^2$, there exists some
$A\subseteq X$ with $|A|\leq\kappa$ such that
$$
\Star{X\times A}{\mathcal{U}} = X^2.
$$

The following is obvious.

\begin{proposition}
For any space $X$, we have $we(X)\leq wee(X)\leq we(X^2)$.
\end{proposition}

By using Example 3.3.4 in \cite{vDreed91}, we are going to
provide a space $X$ such that $we(X)< wee(X)$. Let $\Psi$ be the
Mrowka space $\mathcal A\cup\omega$, where the cardinality of
$\mathcal A$ is $\mathfrak{c}$, and let $Y$ be the one-point
compactification of a
discrete space $D$ of cardinality $\mathfrak{c}$. The space
$X=\Psi\oplus Y$ is the topological sum of a separable space and
a compact space and so we have $we(X)=\omega$. Write $\mathcal
A=\{A_\alpha:\alpha<\mathfrak{c}\}$ and
$D=\{d_\alpha:\alpha<\mathfrak{c}\}$. Let 
\begin{align*}
&U_1=\{\Psi\times \{d_\alpha \}:\alpha <\mathfrak{c}\},\\
&U_2=\{(\{A_\alpha \}\cup A_\alpha )\times (Y\setminus\{d_\alpha \}):\alpha<\mathfrak{c}\}, \\
&U_3=\{\{n\}\times Y:n<\omega\},
\end{align*}

\noindent and finally $\mathcal U=\mathcal
U_1\cup \mathcal U_2 \cup \mathcal U_3 \cup \{Y\times Y\} \cup
\{\Psi\times \Psi\} \cup \{Y\times\Psi\}$.
 
Of course the family
$\mathcal{U}$ is an open cover of $X^2$. Assume that there exists
a countable set $C\subseteq X$ such that $St(X\times C,\mathcal
U)=X^2$. This in turn
would imply the relation $St(\Psi\times (C\cap Y), \mathcal
U_1\cup
\mathcal U_2\cup \mathcal U_3))=\Psi\times Y$.  Since we have
$\Psi\times
Y\setminus (\bigcup \mathcal U_2\cup \bigcup \mathcal
U_3)\supseteq
\{(A_\alpha ,d_\alpha ):\alpha <\mathfrak{c}\}$,  it should be 
$\{(A_\alpha,d_\alpha):\alpha <\mathfrak{c}\}\subseteq
St(\Psi\times (C\cap Y), \mathcal U_1)$. But this  would imply
$D\subseteq C\cap Y$, which  is a contradiction.  This suffices
for the proof  that $wee(X)>\omega=we(X)$.

A further look shows that we actually have $wee(X)=\mathfrak c$.
By repeating the same construction, with the Katetov's extension
in place of $\Psi$ and  with $D$ a set of cardinality
$2^\mathfrak c$, we get a Hausdorff space $X$ such that
$we(X)=\omega$ and $wee(X)=2^\mathfrak c$.

Right now, we do not have a space $X$ for which
$wee(X)<we(X^2)$.

\begin{lemma}
If $wee(X)=\omega $ and  $F$ is a closed subset of $X^2$ and
$\mathcal{U}$ is a cover of $F$ by open subsets of $X^2$, then
there is a countable subset $A$ of $X$ such that
$$
F\subseteq\Star{X\times A}{\mathcal{U}}.
$$
\end{lemma}

\begin{theorem}
If $X$ has a zero-set diagonal and $wee(X)=\omega$, then $X$ is
submetrizable.
\end{theorem}

\begin{proof}
Let $f:X^2\to [0,1]$ be such that $f^{-1}(0)=\Delta_X$. Next, for
$n\in\mathbb{N}$ we let
$C_n=f^{-1}([\nicefrac{1}{n},1])$. Of course $C_n$ is a closed
subset of $X^2$, and $X^2\setminus \Delta_X =
\bigcup_{n\in\mathbb{N}} C_n$.

For $n\in\mathbb{N}$, we let $\mathcal{W}_n$ be defined
by
$$
\mathcal{W}_n = \{ U\times V : U\times V\subseteq
f^{-1}((\nicefrac{1}{2n},1]),~ V\times V\subseteq
f^{-1}([0,\nicefrac{1}{2n}))~\&~ U, V \textnormal{open in}~X\}.
$$
Note that $\mathcal{W}_n$ is a cover of $C_n$ by open subsets of
$X^2$. To see this, fix $n\in\mathbb{N}$ and let $(x,y)\in C_n$. We have $f(x,y)\in (\nicefrac{1}{2n},1]$, and therefore there exist open subsets $U$ and $V$ of $X$ such that $(x,y)\in U\times V\subseteq f^{-1}((\nicefrac{1}{2n},1])$. Moreover, since $(y,y)\in V\times V$ and $f(y,y)=0$ we can shrink $V$ in such a way that $V\times V\subseteq
f^{-1}([0,\nicefrac{1}{2n}))$.

Since $wee(X)=\omega$, by the preceding lemma we may find
a countable subset $B_n$ of $X$ such that
$$
C_n\subseteq \Star{X\times B_n}{\mathcal{W}_n}.
$$
We now let $B=\bigcup_{n\in\mathbb{N}} B_n$, and we define
$F:X\to [0,1]^B$ by
$$
F(x)(b) = f(x,b).
$$
We will show that $F$ is an injection. Since $B$ is countable,
this will imply that $X$ is submetrizable. Pick $x,y\in X$ with
$x\not= y$. Then there is some $n\in\omega\setminus\{0\}$ with
$(x,y)\in C_n$. So we may find $b\in B_n$ and $U\times V\in
\mathcal{W}_n$ such that $(x,y)\in U\times V$ and $b\in V$. Then
$(x,b)\in U\times V$ and $(y,b)\in V\times V$. From the
definition of $\mathcal{W}_n$, it follows that
$$
f(y,b) < \nicefrac{1}{2n} < f(x,b),
$$
and therefore $F(x)\not= F(y)$. This completes the proof.
\end{proof}

The following is the announced generalization of  \cite[Theorem
1]{mar75} and \cite[Theorem 2.1]{buz05}.

\begin{corollary}
If $X^2$ has countable weak extent and a zero-set diagonal, then
$X$ is submetrizable.
\end{corollary}

In \cite[Theorem 2.4 and 2.5]{buz05}, R.Z. Buzyakova proved that
if $X$ has a regular $G_\delta$-diagonal and either it is
separable or $X^2$ has countable extent, then $X$ condenses onto
a second-countable Hausdorff space.

Following the same technique of Buzyakova, we now generalize
those two results.

\begin{theorem}\thlabel{thm:Hausdorff-condensation}
Let  $wee(X)\leq\kappa$ and assume that $X$ has a regular
$G_\delta$-diagonal. Then $X$ condenses onto a Hausdorff space of
weight at most $\kappa$.
\end{theorem}

\begin{proof}
Let $\Delta_X = \bigcap_{n<\omega} U_n = \bigcap_{n<\omega}
\overline{U}_n$, and let $C_n = X^2\setminus U_n$. We define a
family of open sets $\mathcal{U}$ as follows:
$$
\mathcal{U} = \{ U\times V : U\times V\subset
X\setminus\overline{U}_m,V\times V\subset U_m~\mbox{for
some}~m\in\omega~\&~U,V \textnormal{open in}~X\}.
$$
Note that since $\Delta_X = \bigcap_{m\in\omega}\overline{U}_m$,
it follows that $\mathcal{U}$ is an open cover of
$X^2\setminus\Delta_X$.

Since $wee(X)\leq\kappa$, we may find, for every $n\in\omega$, a
subset $B_n$ of $X$ of cardinality at most $\kappa$ such that
$$
C_n\subseteq \Star{X\times B_n}{\mathcal{U}}.
$$
If we let $B = \bigcup_{n\in\omega} B_n$, then $B$ is of
cardinality at most $\kappa$ and
$$
X^2\setminus\Delta_X \subseteq \Star{X\times B}{\mathcal{U}}.
$$
Now we let the family $\mathcal{B}$ consist of all open subsets
of $X$ of one of the following forms:
\begin{enumerate}
\item $\{ y : (y,b)\in U_n \}$ for some $b\in B$ and some
$n\in\omega$,
\item $\{ x : (x,b)\in X^2\setminus \overline{U}_n\}$ for some
$b\in B$ and some $n\in\omega$.
\end{enumerate}

Then since $|B|\leq\kappa$, we also have that
$|\mathcal{B}|\leq\kappa$. We will show that $\mathcal{B}$ is a
Hausdorff separating family (cf. \cite{buz05}).

So, pick $p\not= q$. Then there is some $b\in B$ and $U\times
V\in\mathcal{U}$ such that $b\in V$ and $(p,q)\in U\times V$.
Also, since $U\times V\in\mathcal{U}$, there is some $m\in\omega$
such that
$$
U\times V\subset X\setminus\overline{U}_m~\&~V\times V\subset
U_m.
$$
This means that $(p,b)\in U_m$ and $(q,b)\in
X\setminus\overline{U}_m$, and so we have
\begin{eqnarray*}
p &\in & \{ y : (y,b)\in U_m \} \\
q & \in & \{ x : (x,b)\in X^2\setminus \overline{U}_m\},
\end{eqnarray*}
and since these open sets are disjoint members of $\mathcal{B}$,
this shows that $\mathcal{B}$ is Hausdorff separating.
\end{proof}

\begin{corollary}
If $X^2$ has countable weak extent and a regular
$G_\delta$-diagonal, then $X$ condenses onto a second countable
Hausdorff space.
\end{corollary}

\section{Some cardinal inequalities} 

In this section we prove  various cardinality bounds involving 
different types of diagonal degree. We start off by showing that
for Hausdorff spaces $X$ the inequalities
$\left|X\right|\leq2^{d(X)s\Delta(X)}$ and $\left|X\right|\leq
we(X)^{\Delta_2(X)}$ hold.

Next, we shall prove that if $X$ is either a Baire space with a
rank $2$-diagonal or a space with a rank $3$-diagonal, then its
cardinality is bounded by $wL(X)^\omega$. We do not know if the
same inequality is still true for spaces having a strong rank
$2$-diagonal. However, we can prove that, for such spaces, the
inequality $\left|X\right|\leq wL(X)^{\pi\chi(X)}$ holds. Finally, we will
show
that the  last formula is true for homogeneous spaces having a
regular $G_\delta$-diagonal.

\begin{proposition}
For any Hausdorff space $X$ we have
$$
|X|\leq 2^{d(X)s\Delta(X)}.
$$
\end{proposition}

\begin{proof}
Let $\kappa = d(X) s\Delta(X)$ and fix a family
$\{\mathcal{U}_\alpha : \alpha < \kappa\}$ that witnesses the
fact that $X$ has a strong rank $1$ $G_\kappa$-diagonal. Let $D$
be a dense subset of $X$ of cardinality at most $\kappa$. We
define a map $F:X\to \mathcal{P}(D)^\kappa$ by
$$
F(x)(\alpha) = D\cap \Star{x}{\mathcal{U}_\alpha}.
$$
We only have to show that this map is one-to-one. First of all,
note that since $D$ is dense, we always have $x\in
\overline{F(x)(\alpha)}$. Now let $x\not= y$. Then we may find
$\alpha < \kappa$ with $y\not\in
\overline{\Star{x}{\mathcal{U}_\alpha}}$. But then, since
$F(x)(\alpha)\subseteq \Star{x}{\mathcal{U}_\alpha}$, it follows
that $y\not\in \overline{F(x)(\alpha)}$. So as $y\in
\overline{F(y)(\alpha)}$, it follows that $F(x)(\alpha)\not=
F(y)(\alpha)$.
\end{proof}

One could try  to conjecture the bound $2^{d(X)\Delta(X)}$, but
the Katetov extension of the
discrete space $\omega$ disproves it. It is  separable, it has a
$G_\delta$-diagonal and its cardinality is $2^\mathfrak{c}$.

Taking into account  a result of Ginsburg and Woods, see
\cite[Theorem 9.4]{hod84}, which states that if $X$ is a  $T_1$
space, then its cardinality
is bounded by $2^{e(X)\Delta(X)}$, it is quite natural to wonder
whether the previous proposition can be improved as follows: 
 
\begin{question}
Is the cardinality of a Hausdorff space $X$ bounded by
$2^{we(X)s\Delta(X)}$?
\end{question}

If, in the previous question, we replace
$s\Delta(X)$ with $\Delta_2(X)$, we can actually prove the
following stronger bound.

\begin{proposition}\thlabel{prop:we-rank2}
For any Hausdorff space $X$ we have
$$
|X|\leq we(X)^{\Delta_2(X)}.
$$
\end{proposition}

\begin{proof}
Let $\kappa=we(X)$ and $\lambda=\Delta_2(X)$. Fix a sequence of
open covers $\{\mathcal{U}_\alpha : \alpha < \lambda\}$
witnessing the fact that $X$ has a rank 2 $G_\lambda$-diagonal.
For every $\alpha<\lambda$, we may fix a subset $A_\alpha$ of $X$
with $|A_\alpha|\leq\kappa$ such that $X =
\Star{A_\alpha}{\mathcal{U_\alpha}}$. We let $A=\bigcup_{\alpha <
\lambda} A_\alpha$. Note that $|A|\leq\kappa\cdot\lambda$.

We may fix a map $f:X\to A^\lambda$ with the property that for
$x\in X$ and $\alpha<\lambda$ we have that $f(x)(\alpha) = a\in
A_\alpha$ and $x\in\Star{a}{\mathcal{U}_\alpha}$. To complete the
proof we will show that such a mapping is injective.

So fix $x\not= y$. Then we may find $\alpha < \lambda$ such that
$$
\Star{x}{\mathcal{U}_\alpha}\cap\Star{y}{\mathcal{U}_\alpha} =
\emptyset.
$$
Now let $p=f(x)(\alpha)$. Then
$x\in\Star{p}{\mathcal{U}_\alpha}$, and so also
$p\in\Star{x}{\mathcal{U}_\alpha}$. This means that
$p\not\in\Star{y}{\mathcal{U}_\alpha}$ and therefore
$y\not\in\Star{p}{\mathcal{U}_\alpha}$. This implies that
$p\not=f(y)(\alpha)$. So the mapping $f$ is injective and this
completes the proof.
\end{proof}

This result should be compared with  the inequality $|X|\le
we(X)^{psw(X)}$, obtained by R. Hodel (see \cite{bb06} 
for an alternative and direct proof; see also \cite{hodel91}). The Katetov extension of
$\omega$ witnesses that  in the last two formulas  it is not possible to
put  $\Delta(X)$ at the exponent.  However, one may still try to
conjecture  to improve Ginsburg-Woods' inequality by moving down
$e(X)$ from the exponent. This question was already published by
Bella in 1996 (see \cite{bella96}), but we think is worthy to
repeat it here.

\begin{question} 
Does the inequality
$$|X|\le e(X)^{\Delta(X)}$$ hold for any
$T_1$ space $X$?
\end{question}
  
In \cite[Theorem 2]{bella87}, Bella proved that the cardinality
of a Hausdorff space $X$ is bounded by $2^{c(X)\Delta_2(X)}$.
This was done by an application of the Erd\"os-Rado Theorem. For
Baire spaces with a rank 2-diagonal this bound can be
considerably improved. 

\begin{proposition} \thlabel{baire}
If $X$ a Baire space with a rank 2-diagonal then,
$$
|X|\leq wL(X)^{\omega}.
$$
\end{proposition}

\begin{proof}
This follows from \ref{prop:we-rank2}, the fact that $we(X)\leq
d(X)$ and the following lemma.
\end{proof}

\begin{lemma}
If $X$ is a Baire space with a $G_\delta$-diagonal then,
$$
d(X) \leq wL(X)^{\omega}.
$$
\end{lemma}

\begin{proof}
Let $wL(X)=\kappa$ and let $\{\mathcal{U}_n : n < \omega\}$ be a
sequence of open covers of $X$ witnessing the fact that $X$ has a
rank 1-diagonal. For every $n<\omega$, we fix a family
$\mathcal{V}_n\subseteq\mathcal{U}_n$ of cardinality $\kappa$
whose union is dense in $X$. Next we let $\mathcal{V} =
\bigcup_{n < \omega} \mathcal{V}_n$ and $D_n=
\bigcup\mathcal{V}_n$. Then $|\mathcal{V}|\leq \kappa$, and $D_n$
is an open and dense subset of $X$ for every $n$. Since $X$ is a
Baire space, this means that $D=\bigcap_{n<\omega} D_n$ is a
dense subset of $X$. So to complete the proof it suffices to show
that $|D|\leq \kappa^\omega$. 

We fix some well-ordering on $\mathcal{V}$ and we define a map
$f:D\to \mathcal{V}^\omega$ as follows
$$
f(d)(n) = \min\{ V\in\mathcal{V} : d\in V\in \mathcal{V}_n\}.
$$
We will show that $f$ is an injection. So fix $x,y\in D$ with
$x\not= y$. Then $y\not\in\Star{x}{\mathcal{U}_n}$ for some
$n\in\omega$. Let $V = f(x)(n)$.  Then $x\in V$ and since
$\mathcal{V}_n$ is a refinement of $\mathcal{U}_n$, this means
that $V\subseteq\Star{x}{\mathcal{U}_n}$. So we have that
$y\not\in V$ and therefore $f(x)(n)\not= f(y)(n)$. This completes
the proof.
\end{proof}

We could ask whether the Baire assumption in \ref{baire} is
necessary. This is an open question, but we can prove that for
spaces having a rank $3$-diagonal the following is true.

\begin{proposition}
If $X$ has a rank $3$-diagonal then,
$$
\left|X\right|\leq wL(X)^{\omega}.
$$
\end{proposition}

\begin{proof}
Let $wL(X)=\kappa$ and let $\{\mathcal{U}_n : n<\omega \}$ be a
sequence of open covers of $X$ witnessing the fact that $X$ has a
rank $3$-diagonal. For every $n<\omega$, we fix a family
$\mathcal{V}_n\subseteq \mathcal{U}_n$ of cardinality $\kappa$
whose union is dense in $X$.

Next we let $\mathcal{V}=\bigcup_{n<\omega}\mathcal{V}_n$. Of
course we have $|\sier{V}|\leq wL(X)$. Note that whenever
$U\in\sier{U}_n$, there is some $V\in\sier{V}_n$ such that $U\cap
V\not=\emptyset$. So it follows that for every $x\in X$ and
$n\in\omega$, there is some $V\in\sier{V}_n$ such that
$\st{x}{\sier{U}_n}\cap V\not=\emptyset$. Also note that in this
case $V\subseteq \st[2]{x}{\sier{U}_n}$. We fix a well-ordering
on $\sier{V}$ and we define a map $F : X\to \sier{V}^\omega$ as
follows
$$
F(x)(n) = \min\{ V\in\mathcal{V} :
V\in\mathcal{V}_n~\&~\st{x}{\mathcal{U}_n}\cap V\not=\emptyset\}.
$$
We have just shown that $F$ is well-defined. It remains to show
that $F$ is an injection. So let $x,y\in X$ with $x\not= y$. By
assumption, there is some $n\in\omega$ such that
$$
\st[2]{x}{\mathcal{U}_n} \cap \st{y}{\mathcal{U}_n} = \emptyset.
$$
Since $F(x)(n)\subseteq\st[2]{x}{\sier{U}_n}$ and $F(y)(n)\cap
\st{y}{\sier{U}_n}\not=\emptyset$, it follows that $F(x)(n)\not=
F(y)(n)$. This shows that $F$ is an injection and this completes
the proof.
\end{proof}

The discrete cellularity of a space $X$ is the cardinal number
$dc(X)=\sup\{\left|\mathcal{U}\right|:\mathcal{U}$ is a discrete family of open
subsets of $X\}$.   
The last  result should be compared with the inequality  $|X|\le
2^{dc(X)\Delta_3(X)}$ proved in \cite{bella89}. Note that, at least for regular
spaces, we have $dc(X)\le wL(X)$ and the gap can be artitrarely large.  
We do not know if the last two mentioned inequalities are true for spaces with a
strong rank $2$-diagonal. 


\begin{question}
Let $X$ be a space with a strong rank 2-diagonal. Is it the case
that
\begin{itemize}
\item[$\bullet$] $|X|\leq wL(X)^\omega${?}
\item[$\bullet$] $|X|\le 2^{dc(X)}${?}
\end{itemize}
\end{question}

\noindent However, for  spaces of countable $\pi$-character,  we
have the
answer. 

\begin{proposition}
Let $X$ be a space with a strong rank 2-diagonal. Then
$$
|X|\leq wL(X)^{\pi\chi(X)}.
$$
\end{proposition}

\begin{proof}
Let $\{\mathcal{U}_n : n<\omega\}$ be a sequence of open covers
of $X$ witnessing the fact that $X$ has a strong rank 2-diagonal
and let $\kappa = \pi\chi(X)$ and $\lambda=wL(X)$. For every
$x\in X$, we let $\sier{V}_x = \{ V(x,\alpha) : \alpha <
\kappa\}$ be a local $\pi$-base at $x$. For $n<\omega$, we fix a
family $\sier{W}_n\subseteq\sier{U}_n$ of cardinality $\lambda$
whose union is dense in $X$.

Next we let $\sier{W} = \bigcup_{n<\omega} \sier{W}_n$. Note that
$|\sier{W}|\leq \lambda$. Since $\sier{U}_n$ is a cover of $X$,
it follows that whenever $V$ is a non-empty open subset of $X$,
then $V\cap W\not=\emptyset$ for some $W\in\sier{W}_n$. We fix a
well-ordering on $\sier{W}$ and we define a map $F:X\to
\sier{W}^{\kappa\times\omega}$ as follows,
$$
F(x)(\alpha,n) = \left\{ \begin{array}{ll}
\emptyset, & \mbox{if}~V(x,\alpha)\not\subseteq
\st{x}{\sier{U}_n},\\
\min\{ W\in\sier{W}_n : W\cap V(x,\alpha)\not=\emptyset \}, &
\mbox{otherwise}.
\end{array} \right.
$$
By the remarks made before, the map $F$ is well-defined. For
$x\in X$ and $n<\omega$, we let $W(x,n)$ be defined by
$$
W(x,n) = \bigcup\{ F(x)(\alpha, n) : \alpha \in\kappa\}.
$$
Note that by definition of $F$, we have that $W(x,n)\subseteq
\st{\st{x}{\sier{U}_n}}{\sier{W}_n}$ and since $\sier{W}_n$ is a
refinement of $\sier{U}_n$, it follows that
$$
W(x,n) \subseteq \st[2]{x}{\sier{U}_n}.
$$
\begin{claim} $x\in\closure{W(x,n)}$ for every $n\in \omega$.
\end{claim} 
\begin{pfclaim}
To see this, let $Ox$ be an open neighbourhood of $x$. Then
$V(x,\alpha)\subseteq Ox\cap \st{x}{\sier{U}_n}$ for some $\alpha
< \kappa$. By definition of $F$, it follows that
$F(x)(\alpha,n)\cap V(x,\alpha)\not=\emptyset$ and therefore
$F(x)(\alpha,n)\cap Ox\not=\emptyset$. Since
$F(x)(\alpha,n)\subseteq W(x,n)$, it follows that $x\in
\closure{W(x,n)}$ and this proves the claim.
\end{pfclaim}

So for every $x\in X$, we have that
$$
\{x\}\subseteq \bigcap_{n<\omega} \closure{W(x,n)} \subseteq
\bigcap_{n<\omega} \closure{\st[2]{x}{\sier{U}_n}} = \{x\}.
$$
This shows that $F$ is an injection and this completes the proof.
\end{proof}

For homogeneous spaces, the previous proposition can be improved.

Note that if $X$ is homogeneous and $\pichar{X} = \kappa$, then
there is a collection $\{ V(x,\alpha) : x\in X, \alpha<\kappa\}$
of non-empty open subsets of $X$ such that for every $x\in X$,
$\sier{V}_x = \{ V(x,\alpha) : \alpha<\kappa\}$ is a local
$\pi$-base at $x$ and whenever $Ox$ and $Oy$ are open
neighbourhoods of $x$ and $y$ respectively, there is some
$\alpha<\kappa$ such that
$$
V(x,\alpha)\subseteq Ox~\mbox{and}~V(y,\alpha)\subseteq Oy.
$$
For example, if $p\in X$ is fixed and $\{ V_\alpha : \alpha
<\kappa\}$ is a local $\pi$-base at $p$ in $X$, then we may
define $V(x,\alpha) = h_{x}[V_\alpha]$, where $h_x$ is a
homeomorphism of $X$ mapping $p$ onto $x$. 

\begin{proposition}
Let $X$ be a homogeneous space with a regular
$G_\delta$-diagonal. Then
$$
|X|\leq wL(X)^{\pichar{X}}.
$$
\end{proposition}

\begin{proof}
Fix a sequence $\{\sier{U}_n : n<\omega\}$ of open covers of $X$
witnessing the fact that $X$ has a regular $G_\delta$-diagonal.
Furthermore, let $\pichar{X}=\kappa$ and $wL(X)=\lambda$ and fix
a collection $\{V(x,\alpha) : x\in X, \alpha<\kappa\}$ of
non-empty open subsets of $X$ with the property stated just
before this proposition. 

Next, for $n<\omega$, we fix a family
$\sier{W}_n\subseteq\sier{U}_n $ of cardinality $\lambda$ whose
union is dense in $X$.

Note that since $\sier{U}_n$ is a cover of $X$, if follows that
whenever $V$ is a non-empty open subset of $X$, then $V\cap
W\not=\emptyset$ for some $W\in\sier{W}_n$. We let $\sier{W} =
\bigcup_{n<\omega} \sier{W}_n$ and we fix a well-ordering on
$\sier{W}$. Note that $|\sier{W}|\leq wL(X)$.

We now define a map $F:X\to \sier{W}^{\omega\times\kappa}$ as
follows,
$$
F(x)(n,\alpha) = \min\{ W\in\sier{W} : W\in\sier{W}_n~\&~W\cap
V(x,\alpha)\not=\emptyset\}.
$$
We have just showed that $F$ is well-defined. It remains to
verify that $F$ is an injection, so let $x,y\in X$ with $x\not=
y$. Then there is some $n<\omega$ and open neighbourhoods $Ox$
and $Oy$ of $x$ and $y$ respectively such that
$$
\st{Ox}{\sier{U}_n} \cap Oy =\emptyset.
$$
By the property of our local $\pi$-bases, it follows that there
is some $\alpha<\kappa$ such that
$$
V(x,\alpha)\subseteq Ox~\mbox{and}~V(y,\alpha)\subseteq Oy.
$$
Now recall that $\sier{W}_n$ is a refinement of $\sier{U}_n$, and
therefore, since $V(x,\alpha)\subseteq Ox$,  we have the
following:
$$
F(x)(n,\alpha)\subseteq \st{Ox}{\sier{U}_n}.
$$
Furthermore, by construction we have that $F(y)(n,\alpha)\cap
Oy\not=\emptyset$ so it follows that $F(x)(n,\alpha)\not=
F(y)(n,\alpha)$. This shows that $F$ is an injection and this
completes the proof.
\end{proof}

\bibliography{newreferences}
\bibliographystyle{amsplain}

\end{document}